\documentclass[11pt]{article}

\usepackage[margin=1in]{geometry}
\usepackage{amsmath,amsfonts,amsthm,amssymb}
\usepackage{subcaption}
\usepackage{graphicx}
\graphicspath{{figures/}}
\usepackage{color}
\usepackage{url}
\usepackage{soul}

\usepackage{verbatim}
\usepackage{hyperref}

\newcommand{\bbR}{\mathbb{R}}

\newcommand{\Rn}{\mathbb{R}^n}

\newcommand{\Ldel}{L^{\delta,\beta}}
\newcommand{\Ldelinf}{L^{\delta,-\infty}}

\newcommand{\cdel}{c^{\delta,\beta}}
\newcommand{\mdel}{m^{\delta,\beta}}
\newcommand{\mdelone}{m^{1,\beta}}
\newcommand{\mdelbp}{m^{\delta,\beta'}}
\newcommand{\mdelinf}{m^{\delta,-\infty}}
\newcommand{\udel}{u^{\delta,\beta}}
\newcommand{\hatudel}{\hat{u}^{\delta,\beta}}

\newcommand{\T}{\mathbf{T}}

\DeclareMathOperator{\Res}{Res}

\newtheorem{theorem}{Theorem}
\newtheorem{lemma}{Lemma\textbf{}}
\newtheorem{cor}{Corollary}

\theoremstyle{remark}
\newtheorem{remark}{Remark}

\title{Fourier multipliers for nonlocal  Laplace operators}
\author{Bacim Alali and Nathan Albin\\
\\
{\footnotesize Department of Mathematics, Kansas State University, Manhattan, KS}}

\begin{document}

\maketitle

\begin{abstract}
Fourier multiplier analysis is developed for nonlocal  peridynamic-type Laplace operators, which are defined for scalar fields in $\Rn$. The Fourier multipliers are given through an integral representation. 
We show that the integral representation of the Fourier multipliers is recognized explicitly through a  unified and general formula  in terms of the hypergeometric function $_2F_3$ in any spatial dimension $n$.
Asymptotic analysis of $_2F_3$ is utilized to identify the asymptotic behavior of the Fourier multipliers $m(\nu)$ as $\|\nu\|\rightarrow \infty$. We show that the multipliers are bounded when the peridynamic Laplacian has an integrable kernel, and diverge  when the kernel is singular.  The bounds and decay rates are presented explicitly in terms of the dimension $n$, the integral kernel, and the peridynamic Laplacian nonlocality. The asymptotic analysis is applied in the periodic setting to prove a regularity result for the  peridynamic Poisson equation and, moreover, show that its solution converges to the solution of the classical Poisson equation.
\end{abstract}

{\it Keywords}:
Fourier multipliers, eigenvalues, nonlocal Laplacian, nonlocal equations, peridynamics.

\section{Introduction}
In this work, we study the Fourier multipliers of  nonlocal peridynamic-type Laplace operators and their asymptotic behavior. In  nonlocal vector calculus,  introduced in \cite{nonlocal_calc_2013},  a nonlocal Laplace operator can be defined by
\begin{equation}
\label{eq:general_nonlocal_laplacian}
  L_\gamma u(x) = \int \gamma(x,y) (u(y)-u(x))\;dy,
\end{equation}
 for a symmetric kernel  $\gamma(x,y)=\gamma(y,x)$, with $x,y\in \bbR^n$. For second-rank tensor-valued kernels $\gamma$ and  vector fields $u$, the operator $L_\gamma$ corresponds to a linear peridynamic operator \cite{Silling2000,silling2010linearized}. Various mathematical and engineering studies have addressed linear peridynamics including \cite{silling2016book,du2013analysis,emmrich2007well, ALperid1,mengesha_du_heterog_peridynamics_2014,Alali_Gunzburger0,madenci2014peridynamic}.  
 For scalar-valued kernels $\gamma$ and scalar fields $u$, the operators in \eqref{eq:general_nonlocal_laplacian} have been studied in the context of nonlocal diffusion, digital image correlation, and nonlocal wave phenomena among other applications, see for example \cite{bobaru2010peridynamicheat,burch2011classical,lehoucq2015novel,seleson2013interface}. Several mathematical and numerical studies have focused on nonlocal Laplace operators including  \cite{radu2017nonlocal,du2012analysis,aksoylu2011variational}.

 In this work, we focus on radially symmetric kernels with compact support of the form
 \begin{equation*}\label{eq:kernel}
 \gamma(x,y)=\cdel  \frac{1}{\|y-x\|^\beta}\; \chi_{B_\delta(x)}(y),
 \end{equation*}
where $\chi_{B_\delta(x)}$ is the indicator function of the ball of radius $\delta>0$ centered at $x$, and  the exponent satisfies $\beta<n+2$.
In this case, the nonlocal Laplacian  $\Ldel$, parametrized by the nonlocality parameter $\delta$ and the integral kernel exponent $\beta$, can be written as
\begin{equation}\label{eq:nonlocal_laplacian}
  \Ldel u(x) = \cdel\int_{B_\delta(x)}\frac{u(y)-u(x)}{\|y-x\|^\beta}\;dy
  = \cdel\int_{B_\delta(0)}\frac{u(x+z)-u(x)}{\|z\|^\beta}\;dz,
\end{equation}
where the scaling constant $\cdel$ is 
chosen in such a way that for the function $u(z):=\|z\|^2$, $\Ldel u = 2 n = \Delta u$, or equivalently, it is defined by
\begin{eqnarray}
\label{eq:cdel-integral}
\cdel &:=& \left(\frac{1}{2 n}\int_{B_\delta(0)}\frac{\|z\|^2}{\|z\|^\beta}\;dz\right)^{-1},\\
\label{eq:cdel-explicit}
&=& \frac{2(n+2-\beta)\Gamma\left(\frac{n}{2}+1\right)}
    {\pi^{n/2}\delta^{n+2-\beta}}.
\end{eqnarray}

The Fourier multipliers $\mdel$ of the operator $\Ldel$ are defined through  Fourier transform by
\begin{equation}
\label{eq:multiplier-def}
   \widehat{\Ldel u} = \mdel \; \hat{u}.
\end{equation}
In the special case of periodic domains, the multipliers are in fact the eigenvalues  of $\Ldel$ (see Section \ref{sec:eigenvalues}). Eigenvalues of nonlocal Laplace operators have been recently studied in \cite{du2016asymptotically,du2017fast}.
The work \cite{du2016asymptotically} focused on nonlocal eigenvalues and spectral approximations for a 1-D nonlocal Allen--Cahn (NAC) equation with periodic boundary conditions, and showed that these numerical methods are asymptotically compatible in the sense that they provide convergent approximations
to both nonlocal and local NAC models. The work \cite{du2016asymptotically} provides  numerical examples  corresponding to specific integral kernels for which the  eigenvalues are computed explicitly. A more general method is provided in \cite{du2017fast} for the accurate computations of the  eigenvalues in 1, 2, and 3 dimensions for a class of radially symmetric kernels. The approach in \cite{du2017fast} is based on reformulating the integral representations of the  eigenvalues  as series expansions and as solutions to ODE models, and then developing a hybrid algorithm to compute the eigenvalues that utilizes a truncated series expansion and higher order Runge-Kutta ODE solvers.

In this work, we provide a general and unified approach for the  the analysis, computation, and asymptotics of the Fourier multipliers, in particular the  eigenvalues, of the peridynamic Laplace operator $\Ldel$ in any spatial dimension. 
The main contributions of this work are summarized by the following.
\begin{itemize}
\item In Section \ref{sec:multipliers}, the Fourier multipliers of  $\Ldel$  are identified  through an integral representation
\begin{equation*}\label{eq:multiplier-cosine0}
	\mdel(\nu) = \cdel\int_{B_\delta(0)}\frac{\cos(\nu\cdot z)-1}{\|z\|^\beta}\;dz,
\end{equation*}
for $\beta<n+2$.
\item Theorem \ref{thm:multipliers-2F3} (Section~\ref{sec:hypergeometric}) provides a reformulation  for the multipliers, and in particular for the  eigenvalues of $\Ldel$, in terms of the  hypergeometric function $_2F_3$ for any spatial dimension $n$ and any $\beta<n+2$ 
\begin{equation}
 \label{eq:multiplier-general0}
	\mdel(\nu) = -\|\nu\|^2\,_2F_3\left(1,\frac{n+2-\beta}{2};2,\frac{n+2}{2},\frac{n+4-\beta}{2};-\frac{1}{4}\|\nu\|^2\delta^2\right).
\end{equation}
\item  In Section~\ref{sec:generalized}, we use the fact that formula~\eqref{eq:multiplier-general0} makes sense for all $\beta\in\mathbb{R}\setminus\{n+4,n+6,n+8,\ldots\}$ to extend the definition of the multipliers $\mdel(\nu)$ and consequently the definition of the operator $\Ldel$ acting on  the class of  Schwartz functions through the inverse Fourier transform
\begin{equation*}
\Ldel u(x) = \frac{1}{(2\pi)^n}\int_{\mathbb{R}^n}\mdel(\nu)\hat{u}(\nu)e^{i\nu\cdot x}\;d\nu,
\end{equation*}
to the case  when $\beta\ge n+2$, with $\beta\neq n+4,n+6,n+8,\ldots$. 
\item  In Section~\ref{sec:convergence-laplacian}, we consider two types of limiting behavior for the operator $\Ldel$ and its Fourier multipliers as $\beta\rightarrow n+2$  and as $\delta\rightarrow 0$. 
\begin{itemize}
\item Corollary~\ref{cor:delta-beta} presents convergence of the Fourier multipliers   to the Fourier multipliers of the Laplacian
\[
\mdel(\nu)\rightarrow -\|\nu\|^2, \quad \mbox { for } \nu\in\Rn,
\] 
in both   limits  $\beta\rightarrow n+2$ (for  $\delta>0$) and  $\delta\rightarrow 0^+$ (for  $\beta\in\mathbb{R}\setminus\{n+4,n+6,n+8,\ldots\}$).

\item Theorem~\ref{thm:delta-beta} provides a result on the convergence of the peridynamic Laplacian  to the Laplacian 
\[
\Ldel u(x)\rightarrow \Delta u(x), \quad \mbox { for } u\in C^3(\Rn) \mbox { and  } x\in\Rn,
\] 
in both   limits  $\beta\rightarrow n+2^-$ (for  $\delta>0$) and  $\delta\rightarrow 0^+$ (for  $\beta<n+2$).

\end{itemize}
\item In addition, we consider the case $\beta\rightarrow -\infty$ in Section~\ref{sec:generalized}. 
\begin{itemize}
    \item Theorem~\ref{thm:Ldel-infty} shows that the limiting operator $\Ldelinf$ can be realized as an integral operator, with kernel supported in the $(n-1)$-sphere, which acts on continuous functions $u$ and is given by
\begin{equation*}
    \Ldelinf u(x):=\lim_{\beta\to-\infty}\Ldel u(x) = 
    \frac{2\Gamma\left(\frac{n}{2}+1\right)}{\pi^{n/2}\delta^2}
    \int_{S^{n-1}}\left(u(x+\delta w)-u(x)\right)\textbf{}\;d_{S^{n-1}}V.
\end{equation*}

    \item Theorem~\ref{thm:big-beta-mult} provides explicit formulas for the limiting Fourier multipliers
    \begin{eqnarray*}
   \lim_{\beta\to-\infty}\mdel(\nu) &=&
    \frac{4\Gamma\left(\frac{n}{2}+1\right)}{\delta^2}\left(
    \frac{J_{(n-2)/2}(\|\nu\|\delta)}{\left(\frac{1}{2}\|\nu\|\delta\right)^{(n-2)/2}}
    -\frac{1}{\Gamma\left(\frac{n}{2}\right)}
  \right),\\
  &=&
    -\|\nu\|^2\;\;{}_1F_2\left(1;2,\frac{n+2}{2};-\frac{1}{4}\|\nu\|^2\delta^2\right).
\end{eqnarray*}

\end{itemize}

\item In Section \ref{sec:asymptotics}, we utilize the general formula~\eqref{eq:multiplier-general0} and the asymptotic behavior of the hypergeometric function $_2F_3$ to study the asymptotic behavior of the multipliers $\mdel(\nu)$ as $\|\nu\|\rightarrow \infty$.
Theorem~\ref{thm:asymptotics} characterizes  the asymptotic behavior of the multipliers, and in particular that of the  eigenvalues of $\Ldel$, by different cases for any spatial dimension $n$ and  any kernel exponent $\beta\in\mathbb{R}\setminus\{n+4,n+6,n+8,\ldots\}$.
\begin{itemize}
\item[]
\item {\bf Case 1:} $-\infty<\beta<n$. The multipliers are bounded and satisfy
\[
\mdel(\nu) \sim 
-\frac{2n(n+2-\beta)}{\delta^2(n-\beta)}.
\]

\item {\bf Case 2:} $\beta=n$. The multipliers are unbounded and satisfy
\[
\mdel(\nu) \sim -\frac{2n}{\delta^2}\left(
2\log\|\nu\|+
\log\left(\frac{\delta^2}{4}\right)+\gamma-\psi(\frac{n}{2})\right),
\]
where $\gamma$ is Euler's constant and $\psi$ is the digamma function.
\item {\bf Case 3:} $\beta > n$ with $\beta\neq n+4,n+6,n+8,\ldots$. The multipliers are  unbounded and  satisfy
\[
\mdel(\nu) \sim 
 2\left(\frac{2}{\delta}\right)^{n+2-\beta}
\frac{\Gamma\left(\frac{n+4-\beta}{2}\right)\Gamma\left(\frac{n+2}{2}\right)}{(n-\beta)\Gamma\left(\frac{\beta}{2}\right)}\|\nu\|^{\beta-n}.
\]
In particular, when $\beta=n+2$ the multipliers are explicitly given by $\mdel(\nu)=-\|\nu\|^2$, which is consistent with Corollary~\ref{cor:delta-beta}. 
\item {\bf Case 4:} $\beta=-\infty$.  Theorem~\ref{thm:big-beta-asymptotics}
provides the asymptotic behavior of the limiting multipliers $\mdelinf(\nu):=\lim_{\beta\rightarrow -\infty} \mdel(\nu)$, which are shown to satisfy
\begin{equation*}
\begin{split}
\mdelinf(\nu) =
    \left(\frac{2}{\delta}\right)^{(n+3)/2}
    \frac{\Gamma\left(\frac{n}{2}+1\right)}{\pi^{1/2}}
    \cos\left(
\|\nu\|\delta-\frac{n-1}{4}\pi
\right)\|\nu\|^{-(n-1)/2}\textbf{}
    \\-\frac{2n}{\delta^2}
   + O(\|\nu\|^{-(n+2)/2}).
\end{split}
\end{equation*}
\end{itemize}
\item Periodic analysis for the peridynamic Laplacian is presented in Section \ref{sec:periodic}.
\begin{itemize}
\item Section \ref{sec:eigenvalues} shows that, for  periodic domains,~\eqref{eq:multiplier-general} provides a representation for the eigenvalues of $\Ldel$ and Theorem~\ref{thm:asymptotics} describes their asymptotic behavior. 
\item Theorem~\ref{thm:poisson} provides a regularity result for the periodic nonlocal Poisson equation in any spatial dimension.
\item Theorems~\ref{thm:convergence_solutions} and \ref{thm:convergence_solutions_beta} provide  convergence results for the solution of the nonlocal Poisson equation to the solution of the local  Poisson equation for  $\delta\rightarrow 0^+$ and $\beta\rightarrow (n+2)^-$, respectively.
\end{itemize} 
\end{itemize}

\section{Fourier multipliers}\label{sec:multipliers}

In this section, we derive formulas for the Fourier multipliers of $\Ldel$ utilizing the $_2F_3$ generalized hypergeometric function.  To begin, we express $u(x)$ through its Fourier transform as
\begin{equation*}
	u(x) = \frac{1}{(2\pi)^n}\int_{\mathbb{R}^n}\hat{u}(\nu)e^{i\nu\cdot x}\;d\nu.
\end{equation*}
Since the definition of $\Ldel$ can be extended to the space of tempered distributions through the multipliers derived below, it is sufficient  to assume that $u$ is a Schwartz function.  Applying $\Ldel$ then shows that
\begin{equation*}
\begin{split}
	\Ldel u(x) &=\cdel\int_{B_\delta(0)}\frac{u(x+z)-u(x)}{\|z\|^\beta}\;dz\\
    &= \frac{1}{(2\pi)^n}\int_{\mathbb{R}^n}\left[
    \cdel\int_{B_\delta(0)}\frac{e^{i\nu\cdot z}-1}{\|z\|^\beta}\;dz
    \right]\hat{u}(\nu)e^{i\nu\cdot x}\;d\nu,
\end{split}
\end{equation*}
providing the representation
\begin{equation}\label{eq:multiplier-orig}
	\Ldel u(x) = \frac{1}{(2\pi)^n}\int_{\mathbb{R}^n}\mdel(\nu)\hat{u}(\nu)e^{i\nu\cdot x}\;d\nu,
	\quad\text{where}\quad 
	\mdel(\nu) = \cdel\int_{B_\delta(0)}\frac{e^{i\nu\cdot z}-1}{\|z\|^\beta}\;dz
\end{equation}
are the Fourier multipliers of the operator $\Ldel$.
Since 
\begin{equation*}
	\left|e^{i\nu\cdot z}-1\right| =
    \sqrt{2}\sqrt{1-\cos(\nu\cdot z)} \sim \|\nu\|\|z\|,
\end{equation*}
the integrand in~\eqref{eq:multiplier-orig} is integrable as long as $\beta<n+1$.
The imaginary part of~\eqref{eq:multiplier-orig} vanishes due to symmetry, so an alternative form is
\begin{equation}\label{eq:multiplier-cosine}
	\Ldel u(x) = \frac{1}{(2\pi)^n}\int_{\mathbb{R}^n}\mdel(\nu)\hat{u}(\nu)e^{i\nu\cdot x}\;d\nu,
	\quad\text{where}\quad 
	\mdel(\nu) = \cdel\int_{B_\delta(0)}\frac{\cos(\nu\cdot z)-1}{\|z\|^\beta}\;dz.
\end{equation}
Note that the second integral in (\ref{eq:multiplier-cosine}) is finite for $\beta<n+2$.  Thus, we may consider  the formula (\ref{eq:multiplier-cosine}) as the definition for $\Ldel$ with $\beta<n+2$.
\begin{remark}
Due to symmetry, the nonlocal Laplacian  in (\ref{eq:nonlocal_laplacian}) can also be written as
\begin{equation}\label{eq:Ldel_symmetric}
 \Ldel u(x) = \frac{\cdel}{2} \int_{B_\delta(0)}\frac{u(x+z)+u(x-z) - 2 u(x)}{\|z\|^\beta}\;dz,
\end{equation}
which is well-defined when $\beta<n+2$ for sufficiently differentiable functions $u$  (see Theorem~\ref{thm:delta-beta}).
\end{remark}

\subsection{Hypergeometric series representation}\label{sec:hypergeometric}

The following theorem provides a useful representation of $\mdel(\nu)$ through the generalized hypergeometric function $_2F_3$, defined as (see, e.g.,~\cite[Eq.~(16.2.1)]{NIST:DLMF})
\begin{equation*}
	_2F_3(a_1,a_2;b_1,b_2,b_3;z) :=
	\sum_{k=0}^\infty \frac{(a_1)_k(a_2)_k}{(b_1)_k(b_2)_k(b_3)_k}
	\frac{z^k}{k!},
\end{equation*}
where $(a)_k$ is Pochhammer symbol (the rising factorial).

\begin{theorem}\label{thm:multipliers-2F3}
Let $n\ge 1$, $\delta>0$ and $\beta<n+2$.  The Fourier multipliers $\mdel(\nu)$ in~\eqref{eq:multiplier-cosine} can be written in the form
\begin{equation}\label{eq:multiplier-general}
	m(\nu)=\mdel(\nu) = -\|\nu\|^2\,_2F_3\left(1,\frac{n+2-\beta}{2};2,\frac{n+2}{2},\frac{n+4-\beta}{2};-\frac{1}{4}\|\nu\|^2\delta^2\right).
\end{equation}
\end{theorem}

\begin{proof}
The proof is divided into two cases: $n=1$ and $n\ge 2$.  When $n=1$, and the evenness of cosine implies that
\begin{equation*}
	\int_{B_\delta(0)}\frac{\cos(\nu\cdot z)-1}{\|z\|^\beta}\;dz
    = 2\int_{0}^\delta\frac{\cos(\nu z)-1}{z^\beta}\;dz.
\end{equation*}
Expanding the cosine as a power series and integrating yields
\begin{equation*}
\begin{split}
	(\cdel)^{-1}\mdel(\nu)
    &= 2\sum_{k=1}^\infty\int_{0}^\delta \frac{\left(-\nu^2\right)^{k}}{(2k)!}z^{2k-\beta}\;dz \\
    &= 2\sum_{k=1}^\infty \frac{\left(-\nu^2\right)^{k}}{(2k+1-\beta)(2k)!}\delta^{2k+1-\beta}\\
    &= -2\nu^2\delta^{3-\beta}\sum_{k=1}^\infty\frac{\left(-\nu^2\delta^2\right)^{k-1}}{(2k+1-\beta)(2k)!}\\
    &= -2\nu^2\delta^{3-\beta}\sum_{k=0}^\infty\frac{\left(-\nu^2\delta^2\right)^{k}}{(2k+3-\beta)(2k+2)!}\\
    &= -\frac{2\nu^2\delta^{3-\beta}}{2(3-\beta)}\sum_{k=0}^\infty\frac{2(3-\beta)4^k}{(2k+3-\beta)(2k+2)!}
    \left(-\frac{1}{4}\nu^2\delta^2\right)^{k}\\
    &= -\frac{2\nu^2\delta^{3-\beta}}{2(3-\beta)}\sum_{k=0}^\infty a_k\left(-\frac{1}{4}\nu^2\delta^2\right)^{k},
\end{split}
\end{equation*}
where
\begin{equation*}
	a_k = \frac{2(3-\beta)4^k}{(2k+3-\beta)(2k+2)!}.
\end{equation*}

The series can be recognized as a hypergeometric series with coefficient ratio
\begin{equation*}
\begin{split}
	\frac{a_{k+1}}{a_k} &= 
    \frac{(2k+3-\beta)(2k+2)!4^{k+1}}{(2k+5-\beta)(2k+4)!4^k}\\
    &=\frac{4(2k+3-\beta)(k+1)}{(2k+5-\beta)(2k+4)(2k+3)(k+1)}\\
    &=\frac{\left(k+\frac{3-\beta}{2}\right)(k+1)}{\left(k+\frac{5-\beta}{2}\right)(k+2)\left(k+\frac{3}{2}\right)(k+1)},
\end{split}
\end{equation*}
giving the following representation in terms of the $_2F_3$ hypergeometric function
\begin{equation*}
	(\cdel)^{-1}\mdel(\nu)
    = -\frac{\nu^2\delta^{3-\beta}}{3-\beta}
    \,_2F_3\left(1,\frac{3-\beta}{2};2,\frac{3}{2},\frac{5-\beta}{2};-\frac{1}{4}\|\nu\|^2\delta^2\right).
\end{equation*}
With the choice of $\cdel$ given in~\eqref{eq:cdel-explicit}, it follows that
\begin{equation}\label{eq:multiplier-1}
	\mdel(\nu) = -\nu^2\,_2F_3\left(1,\frac{3-\beta}{2};2,\frac{3}{2},\frac{5-\beta}{2};-\frac{1}{4}\nu^2\delta^2\right),
\end{equation}
which agrees with~\eqref{eq:multiplier-general} when $n=1$.

    In general $n$-dimensional space with $n\ge 2$, it is convenient to align $\nu$ to point in the ``simplest'' direction in spherical coordinates, giving the expression $\nu\cdot z=\|\nu\|r\cos(\phi_1)$. In these coordinates, the multiplier $\mdel(\nu)$ defined in~\eqref{eq:multiplier-cosine} is expressed as
    \begin{equation}\label{eq:big-integral}
    \begin{split}
      (\cdel)^{-1}\mdel(\nu) &=
      \int_0^\delta\frac{1}{r^{\beta-n+1}}
      \int_{S^{n-1}}\left(\cos(\|\nu\|r\cos\phi_1)-1\right)\;d_{S^{n-1}}V\;dr.
    \end{split}
    \end{equation}
The innermost integral can be written as
\begin{equation}\label{eq:inner-S1}
  \int_{S^{n-2}}\int_0^\pi \cos(\|\nu\|r\cos\phi_1)\sin^{n-2}\phi_1\;d\phi_1 d_{S^{n-2}}V - \frac{2\pi^{n/2}}{\Gamma\left(\frac{n}{2}\right)}.
\end{equation}
Applying~\cite[Eq.~(9.1.20)]{a_and_s} shows that
\begin{equation}\label{eq:inner-bessel}
  \int_{S^{n-1}}\left(\cos(\|\nu\|r\cos\phi_1)-1\right)\;d_{S^{n-1}}V
  = 2\pi^{n/2}\left(
    \frac{J_{(n-2)/2}(\|\nu\|r)}{\left(\frac{1}{2}\|\nu\|r\right)^{(n-2)/2}}
    -\frac{1}{\Gamma\left(\frac{n}{2}\right)}
  \right),
\end{equation}
and using~\cite[Eq.~(9.1.10)]{a_and_s} gives the series representation
\begin{equation}\label{eq:inner-series}
  \int_{S^{n-1}}\left(\cos(\|\nu\|r\cos\phi_1)-1\right)\;d_{S^{n-1}}V
  = 2\pi^{n/2}\sum_{k=1}^\infty
  \frac{\left(-\frac{1}{4}\|\nu\|^2r^2\right)^{k}}
  {k!\Gamma\left(\frac{n}{2}+k\right)}.
\end{equation}

Substituting~\eqref{eq:inner-series} into~\eqref{eq:big-integral} shows that
\begin{equation*}
\begin{split}
	(\cdel)^{-1}\mdel(\nu)
    &= 2\pi^{n/2}\sum_{k=1}^\infty\int_0^\delta
    \frac{\left(-\frac{1}{4}\right)^k\|\nu\|^{2k}r^{2k+n-\beta-1}}
  {k!\Gamma\left(\frac{n}{2}+k\right)}\;dr \\
  &= 2\pi^{n/2}\sum_{k=1}^\infty\frac{\left(-\frac{1}{4}\right)^k\|\nu\|^{2k}\delta^{2k+n-\beta}}
  {(2k+n-\beta)k!\Gamma\left(\frac{n}{2}+k\right)} \\
  &= 2\pi^{n/2}\left(-\frac{1}{4}\right)\|\nu\|^2\delta^{2+n-\beta}
  \sum_{k=1}^\infty\frac{\left(-\frac{1}{4}\|\nu\|^{2}\delta^{2}\right)^{k-1}}
  {(2k+n-\beta)k!\Gamma\left(\frac{n}{2}+k\right)}\\
  &= -\frac{\pi^{n/2}\|\nu\|^2\delta^{2+n-\beta}}{2(n+2-\beta)\Gamma\left(\frac{n}{2}+1\right)}
  \sum_{k=0}^\infty a_k\left(-\frac{1}{4}\|\nu\|^{2}\delta^{2}\right)^{k},
\end{split}
\end{equation*}
where
\begin{equation*}\label{eq:ak-coefficients}
	a_k = \frac{(n+2-\beta)\Gamma\left(\frac{n}{2}+1\right)}
  {(2k+2+n-\beta)(k+1)!\Gamma\left(\frac{n}{2}+k+1\right)}.
\end{equation*}
The series can be recognized as a hypergeometric series with coefficient ratio
\begin{equation*}
\begin{split}
	\frac{a_{k+1}}{a_k} &= 
    \frac{(2k+2+n-\beta)(k+1)!\Gamma\left(\frac{n}{2}+k+1\right)}
    {(2k+4+n-\beta)(k+2)!\Gamma\left(\frac{n}{2}+k+2\right)}\\
    &=\frac{\left(k+\frac{n+2-\beta}{2}\right)(k+1)}
    {\left(k+\frac{n+4-\beta}{2}\right)\left(k+\frac{n+2}{2}\right)(k+2)(k+1)},
\end{split}
\end{equation*}
giving the following representation in terms of the $_2F_3$ hypergeometric function.
\begin{equation*}
	(\cdel)^{-1}\mdel(\nu)
    = -\frac{\pi^{n/2}\|\nu\|^2\delta^{2+n-\beta}}{2(n+2-\beta)\Gamma\left(\frac{n}{2}+1\right)}
    \,_2F_3\left(1,\frac{n+2-\beta}{2};2,\frac{n+2}{2},\frac{n+4-\beta}{2};-\frac{1}{4}\|\nu\|^2\delta^2\right).
\end{equation*}
Again, since $\cdel$ is given by~\eqref{eq:cdel-explicit}, the result in~\eqref{eq:multiplier-general} follows.
\end{proof}

\subsection{Extending the definition of $\Ldel$}\label{sec:generalized}

Thus far, we have considered the nonlocal operator $\Ldel$ only for the range of parameter $\beta<n+2$. It is possible to extend the definition of $\Ldel$ to a pseudo-differential operator for  larger $\beta$; $\Ldel$ is simply defined to be the operator with multipliers given by~\eqref{eq:multiplier-general}. The formula given for $\mdel(\nu)$ makes sense for all $\beta$ except those for which the parameter $(n+4-\beta)/2$ is a non-positive integer.  In other words, we may use~\eqref{eq:multiplier-general} to define $\Ldel$ for any $\beta\in\mathbb{R}\setminus\{n+4,n+6,n+8,\ldots\}$.  For $\beta < n+2$, this definition coincides with the integral definition~\eqref{eq:nonlocal_laplacian}.  When $\beta > n+2$, the operator defined through its multipliers is a pseudo-differential operator, and when $\beta=n+2$, we have seen that the $\Ldel$ operator coincides with the Laplacian.  In this section, we consider these different values of $\beta$ more carefully.

\subsubsection{The case $\boldsymbol{\beta}\mathbf{=n+2}$}

If $\beta=n+2$ then~\eqref{eq:multiplier-general} becomes
\begin{equation*}
\mdel(\nu) = -\|\nu\|^2\,_2F_3\left(1,0;2,\frac{n+2}{2},1;-\frac{1}{4}\|\nu\|^2\delta^2\right) = -\|\nu\|^2.
\end{equation*}
In other words, $\Ldel=\Delta$, which is consistent with the results of Section~\ref{sec:convergence-laplacian}.

\subsubsection{The case $\boldsymbol{\beta}\mathbf{\to-\infty}$}

Intuitively, as $\beta\to-\infty$ the most significant contribution from the kernel comes from the boundary of $B_\delta(x)$ and, thus, one might expect that $\Ldel$ can be approximated by a surface integral.  This observation is quantified in the following lemma.

\begin{lemma}\label{lem:big-beta}
Let  $n\ge 1$, $\delta>0$, and  $B_\delta:=B_\delta(0)$.  Suppose $f$ is a continuous function on a larger ball centered at the origin.  Then
\begin{equation}\label{eq:big-beta}
\lim_{\beta\to-\infty}\cdel\int_{B_\delta}\frac{f(z)}{\|z\|^{\beta}}\;dz
= \frac{2\Gamma\left(\frac{n}{2}+1\right)}{\pi^{n/2}\delta^2}
\int_{S^{n-1}}f(\delta w)\;d_{S^{n-1}}V.
\end{equation}
\end{lemma}

\begin{proof}
Writing $z$ in spherical coordinates $z=rw$ with $r\ge 0$ and $w\in S^{n-1}$, we obtain
\begin{equation}\label{eq:big-beta-spherical}
\cdel\int_{B_\delta}f(z)\|z\|^{-\beta} = 
\frac{2\Gamma\left(\frac{n}{2}+1\right)}{\pi^{n/2}\delta^2}
\int_{S^{n-1}}\left[\frac{2+n-\beta}{\delta^{n-\beta}}\int_0^\delta f(rw)r^{n-\beta-1}\;dr\right]\;d_{S^{n-1}}V.
\end{equation}

Fix $\epsilon>0$ and let $\delta'\in(0,\delta)$ be sufficiently close to $\delta$ that $|f(rw)-f(\delta w)|<\epsilon$ for all $w\in S^{n-1}$ and all $r\in[\delta',\delta]$.  Using this choice of $\delta'$, we may rewrite the inner integral of~\eqref{eq:big-beta-spherical} as
\begin{equation}\label{eq:big-beta-inner-split}
\int_0^\delta f(rw)r^{n-\beta-1}\;dr
= \int_0^{\delta'} f(rw)r^{n-\beta-1}\;dr
  + \int_{\delta'}^\delta f(rw)r^{n-\beta-1}\;dr.
\end{equation}

For the first term on the right-hand side of~\eqref{eq:big-beta-inner-split}, observe that
\begin{equation*}
\left|\frac{2+n-\beta}{\delta^{n-\beta}}\int_0^{\delta'}
f(rw)r^{n-\beta-1}\;dr\right|
\le \frac{2+n-\beta}{n-\beta}M\left(\frac{\delta'}{\delta}\right)^{n-\beta},\quad\text{where}\quad
M = \sup_{x\in B_\delta}|f(x)|.
\end{equation*}
The quantity on the right goes to zero as $\beta\to-\infty$.

For the second term on the right-hand side of~\eqref{eq:big-beta-inner-split}, note that
\begin{equation*}
\int_{\delta'}^\delta f(rw)r^{n-\beta-1}\;dr
= f(\delta w)\int_{\delta'}^\delta r^{n-\beta-1}\;dr
+ \int_{\delta'}^\delta(f(rw)-f(\delta w))r^{n-\beta-1}\;dr,
\end{equation*}
which implies that
\begin{equation*}
\begin{split}
\left|\frac{2+n-\beta}{\delta^{n-\beta}}
\int_{\delta'}^\delta f(rw)r^{n-\beta-1}\;dr
- f(\delta w)\right|
\le
M\left|\frac{2}{n-\beta}-\frac{2+n-\beta}{n-\beta}\left(\frac{\delta'}{\delta}\right)^{n-\beta}\right|
\\
+\epsilon\frac{2+n-\beta}{n-\beta}\left(1-\left(\frac{\delta'}{\delta}\right)^{n-\beta}\right).
\end{split}
\end{equation*}
The term on the right approaches $\epsilon$ as $\beta\to-\infty$.

Since $\epsilon>0$ was arbitrary, this implies that
\begin{equation*}
\lim_{\beta\to-\infty}
\frac{2+n-\beta}{\delta^{n-\beta}}\int_0^\delta f(rw)r^{n-\beta-1}\;dr
= f(\delta w).
\end{equation*}
Substituting into~\eqref{eq:big-beta-spherical} gives~\eqref{eq:big-beta}.
\end{proof}

Lemma~\ref{lem:big-beta} provides an integral representation for the limiting operator applied to continuous functions $u$.
\begin{theorem}\label{thm:Ldel-infty}
Let $u\in C(\mathbb{R}^n)$ and let $\delta > 0$.  Then
\begin{equation*}
    \lim_{\beta\to-\infty}\Ldel u(x) = 
    \frac{2\Gamma\left(\frac{n}{2}+1\right)}{\pi^{n/2}\delta^2}
    \int_{S^{n-1}}\left(u(x+\delta w)-u(x)\right)\textbf{}\;d_{S^{n-1}}V.
\end{equation*}
\end{theorem}
\begin{proof}
This is a straightforward application of Lemma~\ref{lem:big-beta} to the function $f(z)=u(x+z)-u(x)$.
\end{proof}

The limiting operator can also be applied to more general functions through the limits of the Fourier multipliers.

\begin{theorem}\label{thm:big-beta-mult}
Let $n\ge 1$ and let $\delta > 0$.  Then
\begin{equation*}
    \lim_{\beta\to-\infty}\mdel(\nu) =
    \frac{4\Gamma\left(\frac{n}{2}+1\right)}{\delta^2}\left(
    \frac{J_{(n-2)/2}(\|\nu\|\delta)}{\left(\frac{1}{2}\|\nu\|\delta\right)^{(n-2)/2}}
    -\frac{1}{\Gamma\left(\frac{n}{2}\right)}
  \right)
  =
    -\|\nu\|^2\;{}_1F_2\left(1;2,\frac{n+2}{2};-\frac{1}{4}\|\nu\|^2\delta^2\right).
\end{equation*}
\end{theorem}

\begin{proof}
The theorem follows from an application of Lemma~\ref{lem:big-beta} to the integral representation~\eqref{eq:multiplier-cosine}.  Since, in this case, $f(z)=\cos(\nu\cdot z)-1$, it follows that
\begin{equation*}
    \lim_{\beta\to-\infty}\mdel(\nu) = \frac{2\Gamma\left(\frac{n}{2}+1\right)}{\pi^{n/2}\delta^2}
    \int_{S^{n-1}}\left(\cos(\delta\nu\cdot w)-1\right)\;d_{S^{n-1}}V.
\end{equation*}
The Bessel function representation comes from substituting $r=\delta$ into~\eqref{eq:inner-bessel}, while the hypergeometric function representation follows from~\eqref{eq:inner-series} (with $r=\delta$):
\begin{equation*}
\begin{split}
    \int_{S^{n-1}}\left(\cos(\delta\nu\cdot w)-1\right)\;d_{S^{n-1}}V
    &= 2\pi^{n/2}\sum_{k=1}^\infty
  \frac{\left(-\frac{1}{4}\|\nu\|^2r^2\right)^{k}}
  {k!\Gamma\left(\frac{n}{2}+k\right)}\\
  &= -\frac{\|\nu\|^2\delta^2\pi^{n/2}}{2\Gamma\left(\frac{n}{2}+1\right)}
  \sum_{k=0}^\infty\frac{\Gamma\left(\frac{n}{2}+1\right)}{(k+1)!\Gamma\left(\frac{n}{2}+k+1\right)}\left(-\frac{1}{4}\|\nu\|^2\delta^2\right)^k.
\end{split}
\end{equation*}
The theorem follows from the observation that the series is a hypergeometric series with ratio of successive coefficients equal to
\begin{equation*}
    \frac{1}{(k+2)\left(k+\frac{n}{2}+1\right)} = 
    \frac{(k+1)}{(k+1)(k+2)\left(k+\frac{n+2}{2}\right)}.
\end{equation*}
\end{proof}

Theorem~\ref{thm:big-beta-mult} provides a natural definition of a limiting operator $\Ldelinf$ defined on a Schwartz function $u$ through the multipliers
\begin{gather}
\Ldelinf u(x) = \int_{\mathbb{R}^n}\mdelinf(\nu)\hat{u}(\nu)e^{i\nu\cdot x}\;d\nu,
	\qquad\text{where}\\
\label{eq:m-infinity}
    \mdelinf(\nu) := \lim_{\beta\to-\infty}\mdel(\nu) = 
    -\|\nu\|^2{}_1F_2\left(1;2,\frac{n+2}{2};-\frac{1}{4}\|\nu\|^2\delta^2\right).
\end{gather}

\subsubsection{The case $\boldsymbol{\beta}\mathbf{\to\infty}$}
Unlike in the previous case, there is no limiting operator as $\beta\to\infty$ (even if we omit the special cases $\beta=n+4,n+6,\ldots$).  This is most easy to see from the asymptotic analysis in the following section, where it is shown that, for $\beta>n$, the Fourier multipliers $\mdel(\nu)$ grow in magnitude as $\|\nu\|^{\beta-n}$.  For large frequencies and large $\beta$, the $\Ldel$ operator behaves similarly to a $(\beta-n)$-order differential operator and, thus, we should not expect to obtain a limit operator as we did for the case $\beta\to-\infty$.
\subsection{Convergence to the Laplacian}\label{sec:convergence-laplacian}
A direct consequence of the characterization of the multipliers  in Theorem~\ref{thm:multipliers-2F3} and the continuity of $_2F_3$ is given by the following result.
\begin{cor}\label{cor:delta-beta}
Let $n\ge 1$ and $\nu\in\Rn$.  Then the  Fourier multipliers $\mdel$  in~\eqref{eq:multiplier-cosine} converge to the Fourier multipliers of the Laplacian  as follows
\[
\lim_{\delta\rightarrow 0^+}\mdel(\nu)=-\|\nu\|^2, \;\;\mbox{ for } \;\; \beta\in\mathbb{R}\setminus\{n+4,n+6,n+8,\ldots\},
\]
and
\[
\lim_{\beta\rightarrow n+2} \mdel(\nu)=-\|\nu\|^2, \;\;\mbox{ for }\;\; \delta>0.
\] 
\end{cor}

In fact, the    multipliers  converge in $C^k$ as given by the next result.
\begin{cor}\label{cor:delta-beta-lp}
Let $n\ge 1$, $k\in\mathbb{N}\cup\{0\}$, and $\beta\in\mathbb{R}\setminus\{n+4,n+6,n+8,\ldots\}$.  Then 
\[
\lim_{\delta\rightarrow 0^+}\mdel=-\|\cdot\|^2, \;\;\mbox{ in } \;\; C^k(M),
\]
and
\[
\lim_{\beta\rightarrow n+2} \mdel=-\|\cdot\|^2, \;\;\mbox{ in } \;\; C^k(M),
\] 
for any compact subset $M$ in $\Rn$.
\end{cor}
\begin{proof}
We note that Theorem~\ref{thm:multipliers-2F3} implies that
\[
\mdel(\nu)-(-\|\nu\|^2)= \delta^2 (n+2-\beta) g^{\delta,\beta}(\nu),
\]
for some $C^\infty$ function $g^{\delta,\beta}$. Moreover, there is a constant $B_k$ such that
\[
\| \mdel(\nu)-(-\|\nu\|^2)\|_{C^k(M)}\le \delta^2 (n+2-\beta) B_k,
\]
from which the result follows.
\end{proof}

Convergence of peridynamic operators (in both  scalar and vector cases) to the corresponding differential operators in the limit of vanishing nonlocality $\delta\rightarrow 0^+$ is well-know, see for example \cite{emmrich2007well,nonlocal_calc_2013,mengesha_du_heterog_peridynamics_2014, Alali_Gunzburger0,radu2017nonlocal}. Corollary~\ref{cor:delta-beta} recovers the convergence  $\Ldel\rightarrow \Delta$, as $\delta\rightarrow 0^+$, but in the sense of convergence of the multipliers. Moreover, Corollary~\ref{cor:delta-beta} provides a new result on the convergence of the nonlocal multipliers to the local ones in the limit  as $\beta\rightarrow (n+2)^-$. This result, motivates one to provide an analogous new result on the  convergence of the nonlocal Laplacian to the Laplacian  in the limit  as $\beta\rightarrow (n+2)^-$. For completeness of the presentation and due to similarity of the proofs, we also include the case $\delta\rightarrow 0^+$.
\begin{theorem}\label{thm:delta-beta}
Let $n\ge 1$,  $x\in \Rn$, and $u\in C^3(\Rn)$.  Then
\[
\lim_{\delta\rightarrow 0^+}\Ldel u(x)=\Delta u(x), \;\;\mbox{ for } \;\;  \beta<n+2,
\]
and
\[
\lim_{\beta\rightarrow n+2^-}\Ldel u(x)=\Delta u(x), \;\;\mbox{ for }\;\; \delta>0.
\] 
\end{theorem}
\begin{proof}
Note that the integral in~\eqref{eq:Ldel_symmetric} is well-defined for $u\in C^3(\Rn)$ and $\beta<n+2$. Using Taylor's theorem, we expand $u$ about $z=x$,
\begin{equation}\label{eq:taylor1}
u(x \pm z) = u(x) \pm \frac{\partial u(x)}{\partial x_i} z_i + \frac{1}{2} \frac{\partial^2 u(x)}{\partial x_i \partial x_j} z_i z_j + R(u; x, \pm z),
\end{equation}
where
\begin{equation}\label{eq:taylor2}
R(u; x, \pm z) =  \pm \frac{1}{6} \frac{\partial^3 u(x \pm s z)}{\partial x_i \partial x_j \partial x_k } z_i z_j z_k,
\end{equation}
for some $s\in[0,1]$. Here we are adopting the summation convention and the indices run over $1, 2,  \ldots, n$. Substituting~\eqref{eq:taylor1} and \eqref{eq:taylor2} in~\eqref{eq:Ldel_symmetric}, one obtains
\begin{eqnarray}\label{eq:Ldel_taylor1}
\nonumber
\Ldel u(x) &=& \cdel  \int_{B_\delta(0)} \frac{z_i z_j}{\|z\|^\beta}\;dz \; \frac{1}{2} \frac{\partial^2 u(x)}{\partial x_i \partial x_j} \\
&& + \;\cdel  \int_{B_\delta(0)} \frac{z_i z_j z_k}{\|z\|^\beta} \; 
\frac{1}{6}\left( \frac{\partial^3 u(x+s z)}{\partial x_i \partial x_j\partial x_k} -\frac{\partial^3 u(x-s z)}{\partial x_i \partial x_j\partial x_k} \right)\;dz.
\end{eqnarray}
Using~\eqref{eq:cdel-integral} and the symmetry of the integral, it follows that
\[
\cdel  \int_{B_\delta(0)} \frac{z_i z_j}{\|z\|^\beta}\;dz = \
\begin{cases} 1 &\mbox{for } i=j \\ 
0 & \mbox{for } i \ne j 
\end{cases},
\]
and hence the  first term on the right hand side of~\eqref{eq:Ldel_taylor1} reduces to $\Delta u(x)$. Therefore,
\begin{eqnarray}\label{eq:Ldel-Lap}
\nonumber
\left | \Ldel u(x) - \Delta u(x)\right | &=&   \left | \cdel  \int_{B_\delta(0)} \frac{z_i z_j z_k}{\|z\|^\beta} \; 
\frac{1}{6}\left( \frac{\partial^3 u(x+s z)}{\partial x_i \partial x_j\partial x_k} -\frac{\partial^3 u(x-s z)}{\partial x_i \partial x_j\partial x_k} \right)\;dz \right |, \\
&\leq&  \cdel  \int_{B_\delta(0)} \frac{\|z\|^3}{\|z\|^\beta} \;dz \;H_x,
\end{eqnarray}
where
\[
H_x = \max_{\substack{z \in\overline{B_\delta(0)} \\ s\in[0,1]\\i,j,k=1\ldots n}} \frac{1}{6} \left | \frac{\partial^3 u(x+s z)}{\partial x_i \partial x_j\partial x_k} -\frac{\partial^3 u(x-s z)}{\partial x_i \partial x_j\partial x_k}\right |.
\]
Since
\begin{eqnarray*}
\int_{B_\delta(0)} \frac{\|z\|^3}{\|z\|^\beta} \;dz &=& \int_{S^{n-1}}\; d_{S^{n-1}}V\; \int_0^\delta \frac{r^3}{r^\beta} r^{n-1}\;dr,\\
&=& \frac{2 \pi^{n/2}}{\Gamma(\frac{n}{2})}\; \frac{\delta^{3+n-\beta}}{3+n-\beta},
\end{eqnarray*}
and by using ~\eqref{eq:cdel-explicit} and~\eqref{eq:Ldel-Lap}, one obtains
\[
\left | \Ldel u(x) - \Delta u(x)\right | \leq H_x \frac{2 n (2+n-\beta)}{3+n-\beta} \;\delta,
\] 
from which the result follows.
\end{proof}

\section{Asymptotic behavior of the multipliers}\label{sec:asymptotics}

The series representation~\eqref{eq:multiplier-general} immediately provides the asymptotic behavior of $\mdel(\nu)$ for small $\nu$.
\begin{theorem}\label{thm:asymptotic-near-zero}
Let $n\ge 1$, $\delta>0$ and $\beta\in\mathbb{R}\setminus\{n+2,n+4,n+6,\ldots\}$.  Then, as $\|\nu\|\to 0$,
\begin{equation*}
\frac{\mdel(\nu)}{-\|\nu\|^2} = 1-\frac{\delta^2(n+2-\beta)}{4(n+2)(n+4-\beta)}\|\nu\|^2 + O(\|\nu\|^4).
\end{equation*}
\end{theorem}
\begin{proof}
This is an immediate consequence of the series representation:
\begin{equation*}
\mdel(\nu) = -\|\nu\|^2\left(1-\frac{\delta^2(n+2-\beta)}{4(n+2)(n+4-\beta)}\|\nu\|^2 + O(\|\nu\|^4)\right).
\end{equation*}
\end{proof}

In order to understand the behavior of $\mdel(\nu)$ for large $\|\nu\|$, we approximate $_2F_3(1,a;2,b,a+1;-z^2)$ for large $z=\|\nu\|\delta/2>0$, where $a = (n+2-\beta)/2$ and $b = (n+2)/2$. We establish the asymptotic behavior $\mdel(\nu)$ for large $\|\nu\|$ using formulas from the NIST DLMS~\cite{NIST:DLMF}.

\begin{theorem}\label{thm:asymptotics}
Let $n\ge 1$, $\delta>0$ and $\beta\in\mathbb{R}\setminus\{n+2,n+4,n+6,\ldots\}$.  Then, as $\|\nu\|\to\infty$,
\begin{equation*}
\mdel(\nu) \sim 
\begin{cases}
-\frac{2n(n+2-\beta)}{\delta^2(n-\beta)}
+ 2\left(\frac{2}{\delta}\right)^{n+2-\beta}
\frac{\Gamma\left(\frac{n+4-\beta}{2}\right)\Gamma\left(\frac{n+2}{2}\right)}{(n-\beta)\Gamma\left(\frac{\beta}{2}\right)}\|\nu\|^{\beta-n}
&\text{if $\beta\ne n$},\\
-\frac{2n}{\delta^2}\left(
2\log\|\nu\|+
\log\left(\frac{\delta^2}{4}\right)+\gamma-\psi(\frac{n}{2})\right)
&\text{if $\beta = n$},
\end{cases}
\end{equation*}
where $\gamma$ is Euler's constant and $\psi$ is the digamma function.
\end{theorem}

\begin{proof}
Specialized to the present case,~\cite[Eq.~(16.11.8)]{NIST:DLMF} states that for large $|z|$
\begin{equation*}
_2F_3(1,a;2,b,a+1;-z^2) 
\sim a\Gamma(b)\left(H_{2,3}(z^2) + E_{2,3}(z^2e^{-i\pi}) + E_{2,3}(z^2e^{i\pi})\right),
\end{equation*}
where $E_{2,3}$ and $H_{2,3}$ are formal series defined in~\cite[Eq.~(16.11.1)]{NIST:DLMF} and~\cite[Eq.~(16.11.2)]{NIST:DLMF} respectively.  Again, specializing to the present setting yields
\begin{equation*}
E_{2,3}(z^2e^{\pm i\pi}) = (2\pi)^{-1/2}2^{b+1}e^{\pm 2iz}
(\pm 2iz)^{-\frac{2b+3}{2}}.
\end{equation*}
Since $2b=n+2$, the $E_{2,3}$ terms decay asymptotically like $|z|^{-\frac{n+5}{2}}$, and do not contribute the the asymptotic behavior described in the theorem.

Perhaps the simplest way to obtain the asymptotic behavior of the term involving $H_{2,3}$ is through the remark below~\cite[Eq.~(16.11.5)]{NIST:DLMF}, which states that $H_{2,3}$ can be recognized as the sum of the residues of certain poles of the integrand in~\cite[Eq.~(16.5.1)]{NIST:DLMF}.  Although, in the general case, there are infinitely many poles to consider, the particular configuration of the parameters under consideration provides a simplification of the integrand (keeping in mind we are evaluating $_2F_3$ at $-z^2$):
\begin{equation*}
f(s) := \frac{\Gamma(1+s)\Gamma(a+s)}{\Gamma(2+s)\Gamma(b+s)\Gamma(a+1+s)}
\Gamma(-s)(z^2)^s = \frac{\Gamma(-s)}{(s+1)(s+a)\Gamma(b+s)}z^{2s}.
\end{equation*}
The two poles of interest are at $s=-1$ and $s=-a$.  The restriction that $\beta\notin\{n+2,n+4,n+6,\ldots\}$ ensures that $-a$ is not a nonnegative integer and, therefore, that $s=-a$ is not a pole of $\Gamma(-s)$.  Thus, there are only two cases to consider, either $\beta\ne n$, which implies that $s=-1$ and $s=-a$ are distinct simple poles of $f$, or $\beta=n$, yielding a double pole at $s=-1$.

In the first case, if $\beta\ne n$, we have
\begin{equation*}
H_{2,3}(z^2) = \Res(f,-1) + \Res(f,-a)
= \frac{1}{(a-1)\Gamma(b-1)}z^{-2} + \frac{\Gamma(a)}{(1-a)\Gamma(b-a)}z^{-2a}.
\end{equation*}
Recalling~\eqref{eq:multiplier-general} and substituting $z=\|\nu\|\delta/2$, $a=(n+2-\beta)/2$ and $b=(n+2)/2$ yields the $\beta\ne n$ case of the theorem.

When $\beta=n$, the double pole makes the residue more complicated (and gives rise to the logarithmic term):
\begin{equation*}
\Res(f,-1) = \left.\frac{d}{ds}\left(\frac{\Gamma(-s)}{\Gamma(b+s)}z^{2s}\right)\right|_{s=-1} = \frac{2\log z-\psi(b-1)-\psi(1)}{\Gamma(b-1)}z^{-2}.
\end{equation*}
Again using~\eqref{eq:multiplier-general} and substituting yields the $\beta=n$ case.
\end{proof}

\begin{figure}
\centering
\includegraphics[width=0.95\textwidth]{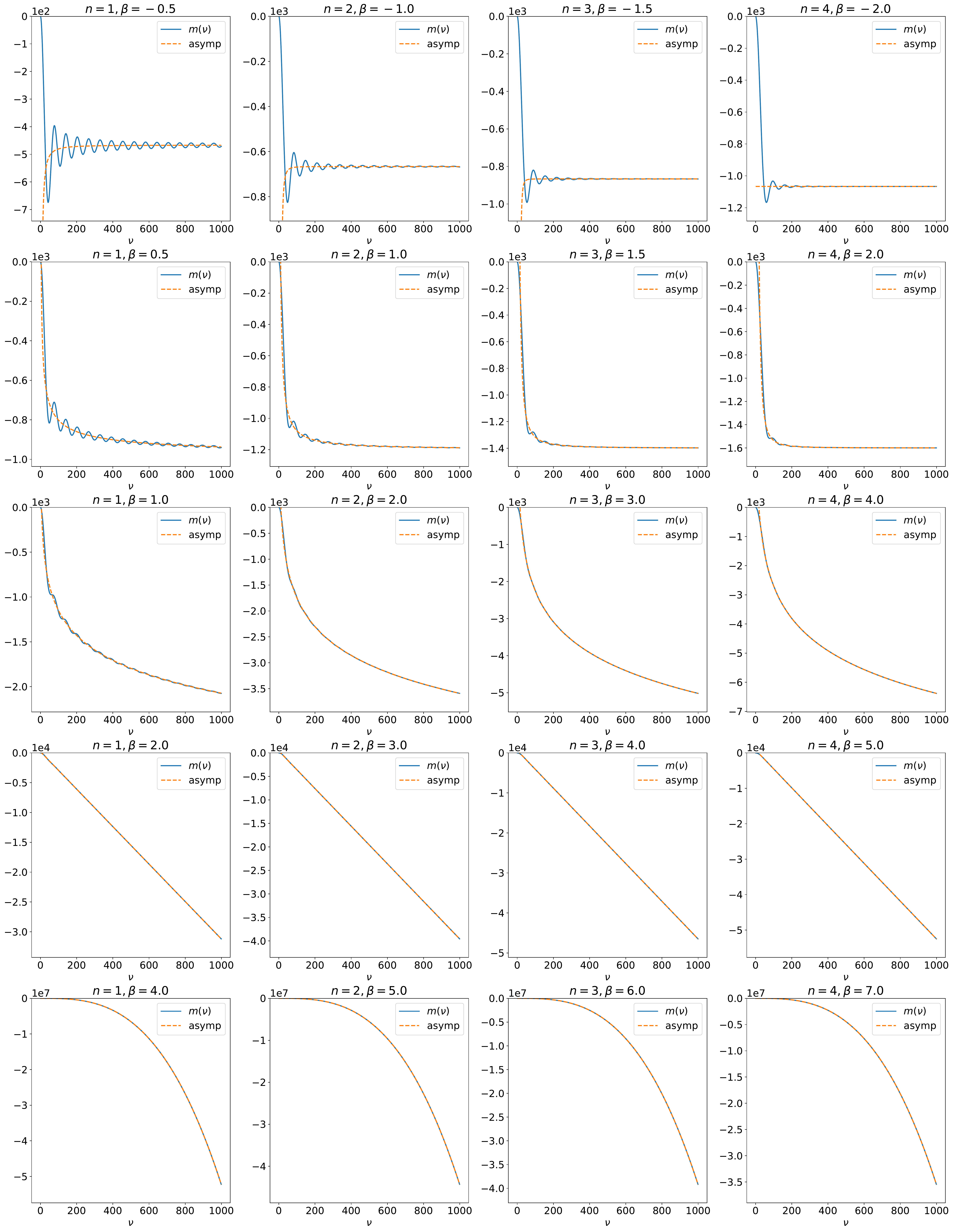}
\caption{Fourier multipliers $\mdel(\nu)$ with $\|\nu\|$ sampled at $1000$ equispaced points in the interval $[1,318\pi]$ and $\delta=0.1$. The values of $n$ and $\beta$ are displayed above each plot.  The dashed lines show the asymptotic approximation from Theorem~\ref{thm:asymptotics}.}
\label{fig:asymptotics-1}
\end{figure}

Figure~\ref{fig:asymptotics-1} shows a comparison between the actual multipliers and the asymptotic approximation in Theorem~\ref{thm:asymptotics} for various cases.  The values of the multipliers were evaluated using the implementation of the $_2F_3$ hypergeometric functions that are provided by mpmath~\cite{mpmath}, a Python library for arbitrary precision arithmetic.  The fact that the solid curves (computed multipliers) approach the dashed curves (asymptotic behavior given in Theorem~\ref{thm:asymptotics}) as $\nu$ grows large in all cases demonstrates the expected asymptotic behavior of the multipliers.  In the first two rows of plots, $\beta<n$, so Theorem~\ref{thm:asymptotics} implies that the multipliers are bounded.  The third row of plots corresponds to $\beta=n$, for which the theorem predicts logarithmic growth in the magnitudes of the multipliers.  In the fourth row, $\beta=n+1$ and the theorem predicts asymptotic linear growth.  In the final row, $\beta=n+3$ (a case of the extension described in Section~\ref{sec:generalized}).  Again, one can observe that the multipliers asymptotically approach the predicted curve, which exhibits cubic growth in this case.

\begin{remark}
In light of Section~\ref{sec:eigenvalues}, Theorem~\ref{thm:asymptotics} provides a generalization to the key estimates within the proof of~\cite[Lemma~3]{du2016asymptotically}.  Rewriting  these estimates in the notation of the present paper yields the following statement.  In the 1D case ($n=1$), there exist positive constants $C_1(\delta),C_2(\delta),\ldots,C_6(\delta)$ such for sufficiently large $|\nu|$,
\begin{equation*}
\begin{cases}
C_1(\delta) \le |\mdel(\nu)| \le C_2(\delta) &\text{if $0<\beta<1$}, \\
C_3(\delta)|\nu|^{\beta-1} \le |\mdel(\nu)| \le C_4(\delta)|\nu|^{\beta-1}
&\text{if $1<\beta<3$},\\
C_5(\delta)\log(|\nu|) \le |\mdel(\nu)| \le C_6(\delta)\log(|\nu|)
&\text{if $\beta=1$}.
\end{cases}
\end{equation*}
Theorem~\ref{thm:asymptotics} shows that similar estimates exist in arbitrary spatial dimension and, moreover, provides explicit coefficients for the asymptotic behavior as $\|\nu\|\to\infty$ including in cases not traditionally considered ($\beta\ge n+2$).
\end{remark}

The asymptotic behavior of the limit operator $\Ldelinf$, defined through its multipliers in~\eqref{eq:m-infinity} can be given a more explicit form.
\begin{theorem}\label{thm:big-beta-asymptotics}
Let $n\ge 1$ and $\delta>0$.  Then,
\begin{equation}\label{eq:asysmtotics-near-zero-big-beta}
\frac{\mdelinf(\nu)}{-\|\nu\|^2} = 1 - \frac{\delta}{4(n+2)}\|\nu\|^2 + O(\|\nu\|^4)
\end{equation}
and
\begin{equation}\label{eq:asymptotics-big-beta}
\begin{split}
\mdelinf(\nu) =
    \left(\frac{2}{\delta}\right)^{(n+3)/2}
    \frac{\Gamma\left(\frac{n}{2}+1\right)}{\pi^{1/2}}
    \cos\left(
\|\nu\|\delta-\frac{n-1}{4}\pi
\right)\|\nu\|^{-(n-1)/2}\textbf{}
    \\-\frac{2n}{\delta^2}
   + O(\|\nu\|^{-(n+2)/2}).
\end{split}
\end{equation}
\end{theorem}
\begin{proof}
The asymptotic formula~\eqref{eq:asysmtotics-near-zero-big-beta} follows from the series expansion of~\eqref{eq:m-infinity} near $\nu=0$.

From~\cite[Eq.~(9.2.1)]{a_and_s},
\begin{equation*}
J_{(n-2)/2}(\|\nu\|\delta) = 
\sqrt{\frac{2}{\pi\|\nu\|\delta}}
\cos\left(
\|\nu\|\delta-\frac{n-1}{4}\pi
\right) + O(\|\nu\|^{-3/2}).
\end{equation*}
Substituting this into the Bessel function representation in Theorem~\ref{thm:big-beta-mult} yields~\eqref{eq:asymptotics-big-beta}.

\end{proof}

\begin{figure}
\centering
\includegraphics[width=0.95\textwidth]{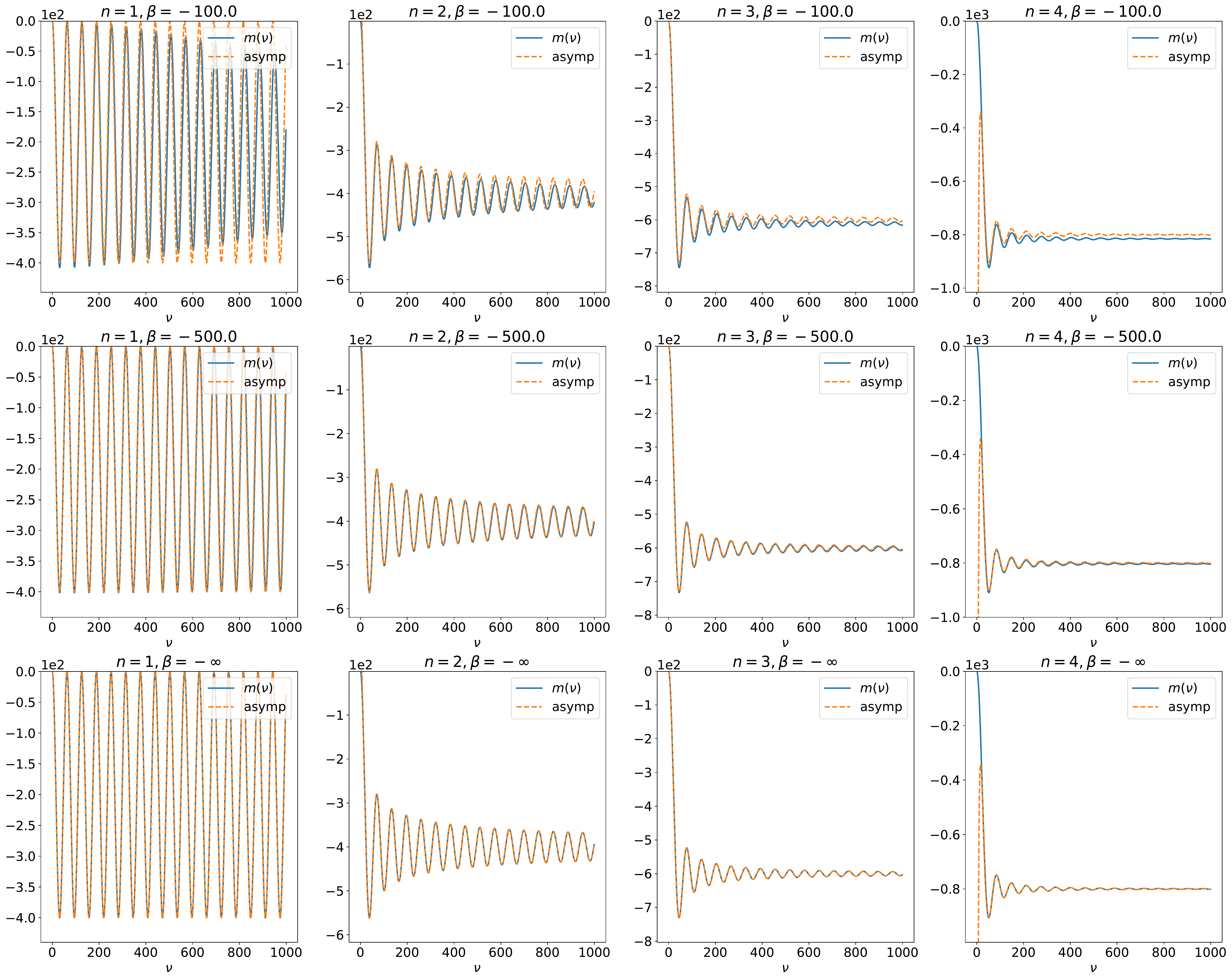}
\caption{Fourier multipliers $\mdel(\nu)$ with $\|\nu\|$ sampled at $1000$ equispaced points in the interval $[1,318\pi]$ and $\delta=0.1$. The values of $n$ and $\beta$ are displayed above each plot.  The dashed lines show the asymptotic approximation of $\mdelinf(\nu)$ from Theorem~\ref{thm:big-beta-asymptotics}.}
\label{fig:asymptotics-2}
\end{figure}

Figure~\ref{fig:asymptotics-2} shows a comparison between the actual multipliers and the asymptotic approximation in Theorem~\ref{thm:big-beta-asymptotics}.  The values of the multipliers were evaluated using the implementation of the $_2F_3$ and $_1F_2$ hypergeometric functions that are provided by mpmath~\cite{mpmath}.  Each row of the figure corresponds to a particular value of $\beta$ ($-100$, $-500$ and $-\infty$ respectively).  The solid curves show the values of the multipliers while the dashed curves show the values of the right-hand side of~\eqref{eq:asymptotics-big-beta} with error terms omitted.  Note that, for $n\ge 2$, Theorem~\ref{thm:big-beta-asymptotics} predicts that $\mdelinf(\nu)\to-\frac{2n}{\delta^2}$ as $\nu\to\infty$.  On the other hand, for finite negative $\beta$ with large magnitude, Theorem~\ref{thm:asymptotics} implies that
\begin{equation*}
\mdel(\nu) \to -\frac{2n}{\delta^2}\cdot\frac{n+2-\beta}{n-\beta} = -\frac{2n}{\delta^2}\left(1 + O\left(\frac{1}{\beta}\right)\right).
\end{equation*}
Thus, for $n\ge 2$, the asymptotic behaviors of $\mdel$ for $\beta=-\infty$ and $-\beta\gg1$ are nearly identical; the multipliers in both cases approach nearly the same constant.  However, when $n=1$, the behaviors are quite different.  For large magnitude negative $\beta$, the multipliers approach a constant near $-\frac{2}{\delta^2}$.  However, for $\beta=-\infty$, the multipliers oscillate with finite amplitude and do not approach a limit as $\|\nu\|\to\infty$.

\section{Periodic analysis}\label{sec:periodic}
In  this section, we present an analog to the analysis presented in Sections~\ref{sec:multipliers} and~\ref{sec:asymptotics}   when $\Ldel$ is treated as an operator on periodic functions. In particular, we show that~\eqref{eq:multiplier-general} provides a representation for the eigenvalues of $\Ldel$ and Theorem~\ref{thm:asymptotics} describes their asymptotic behavior. Moreover, we utilize  Theorems~\ref{thm:multipliers-2F3} and \ref{thm:asymptotics} to prove a regularity result for the peridynamic Poisson equation. Furthermore, we show that the solution of the peridynamic Poisson equation converges to the solution of the classical Poisson equation as $\delta \rightarrow 0^+$.

\subsection{Eigenvalues on periodic domains}\label{sec:eigenvalues}
Consider $\Ldel$ as an operator on the periodic torus
\begin{equation*}
\T^n = \prod_{i=1}^n[0,\ell_i],\qquad\text{with }\ell_i>0,\quad i=1,2,\ldots,n.
\end{equation*}

For any $k \in\mathbb{Z}^n$, define
\begin{eqnarray}\label{eq:nu}
\nu_k&=&(2\pi k_1/ \ell_1,2\pi k_2/ \ell_2,\ldots,2\pi k_n/ \ell_n)^T,\\
\nonumber
\phi_k(x) &=& e^{i\nu_k\cdot x}.
\end{eqnarray}
The functions $\{\phi_k\}_{k\in\mathbb{Z}^n}$ form a complete set in $L^2(\T^n)$.  Moreover,
\begin{equation}\label{eq:phi}
  \Ldel\phi_k(x) =
  \left(
  \cdel\int_{B_\delta(0)}\frac{e^{i\nu_k\cdot z}-1}{\|z\|^\beta}\;dz
  \right)\phi_k(x)
  = m(\nu_k)\phi_k(x),
\end{equation}
implying that $\phi_k$ is an eigenfunction of $\Ldel$ with eigenvalue $m(\nu_k)$.  Thus,~\eqref{eq:multiplier-general} provides an alternative to the integral representation of the eigenvalues in the periodic setting and, 
consequently, Theorem~\ref{thm:asymptotics} describes their asymptotic behavior.

\subsection{Regularity of solutions for the peridynamic Poisson equation}\label{sec:scalar-reg-per}

Consider  the periodic peridynamic Poisson equation $\Ldel u = f$.  For $s\in\mathbb{R}$, let $H^s(\T^n)$ consist of the periodic distributions $g$ on $\mathbf{T}^n$ with the property that
\begin{equation*}
    \sum_{k\in\mathbb{Z}^n}(1+\|k\|^2)^s\|\hat{g}_k\|^2
    < \infty.
\end{equation*}
Using the asymptotic properties of the eigenvalues, we prove the following generalization of~\cite[Lemma~3]{du2016asymptotically}.
\begin{theorem}\label{thm:poisson}
Let $n\ge 1$, $\delta>0$ and $\beta \le n+2$.  If $f\in H^s(\T^n)$ satisfies $\hat{f}_0 = 0$, then there is a unique $u\in H^{s'}(\T^n)$ satisfying $\Ldel u = f$ and $\hat{u}_0=0$, where $s'=s + \max\{0,\beta-n\}$.
\end{theorem}
\begin{proof}
Let  $f\in  H^s(\T^n)$ be represented through its Fourier series. Define 
\begin{eqnarray}\label{eq:u_k}
\nonumber
\hat{u}_0 &=& 0,\\
\hat{u}_k &=& \frac{1}{m(\nu_k)} \hat{f}_k, \;\mbox{ for }\; k\in \mathbb{Z}^n\setminus \{0\},
\end{eqnarray}
where the $\hat{f}_k$ are the Fourier coefficients of $f$, $\nu_k$ are defined by~\eqref{eq:nu}, and $m$ is given in~\eqref{eq:multiplier-general}. Define 
\[
u(x) := \sum_{k\in\mathbb{Z}^n}\hat{u}_k e^{i \nu_k\cdot x}.
\] 
Then using~\eqref{eq:phi},~\eqref{eq:u_k}, and the fact that $\hat{f}_0=0$, it follows that
\begin{equation*}\label{eq:Ldel-series}
    \Ldel u(x) = \sum_{k\in\mathbb{Z}^n}m(\nu_k)\hat{u}_k e^{i \nu_k\cdot x}=f(x).
\end{equation*}
It remains to show that $u\in H^{s'}(\T^n)$. We observe that
\begin{equation}\label{eq:Hs1}
\sum_{0\neq k\in \mathbb{Z}^n}(1+\|k\|^2)^{s'}\|\hat{u}_k\|^2 = \sum_{0\neq k\in \mathbb{Z}^n}\frac{(1+\|k\|^2)^{s'-s}}{|m(\nu_k)|^2}\;(1+\|k\|^2)^{s} \|\hat{f}_k\|^2.
\end{equation}
From \eqref{eq:Hs1} and since $f\in H^s(\T^n)$, the result follows by showing that 
\[
\frac{(1+\|k\|^2)^{s'-s}}{|m(\nu_k)|^2}
\]
is bounded for $k\neq 0$. To see this, we consider two cases. 
First, for $\beta\leq  n$, then $s'-s=0$ in this case, and by using Theorem~\ref{thm:asymptotics}, there exist $r_1>0$ and $C_1>0$ such that $|m(\nu_k)|\ge C_1$, for all $\|k\|\ge r_1$. This implies that 
\[
\frac{(1+\|k\|^2)^{s'-s}}{|m(\nu_k)|^2} \leq \frac{1}{C_1^2}.
\]
Similarly, for $\beta>n$, then $s'-s=\beta-n$, and by using Theorem~\ref{thm:asymptotics}, there exist $r_2>0$ and $C_2>0$ such that $|m(\nu_k)|\ge C_2 \|k\|^{\beta-n}$, for all $\|k\|\ge r_2$. This implies that 
\[
\frac{(1+\|k\|^2)^{s'-s}}{|m(\nu_k)|^2} \leq \frac{1}{C_2^2} \left(\frac{1+\|k\|^2}{\|k\|^2}\right)^{\beta-n},
\]
which is bounded, completing the proof.
\end{proof}

Next we provide a result on the convergence of solutions of the nonlocal problem to the solution of the local problem.

\begin{lemma}\label{lem:reciprocals-bounded}
Let $n\ge 1$ and $\beta\le n+2$.  Define $\theta=\max\{0,\beta-n\}$.  Then there exists a constant $C>0$ such that
\begin{equation*}
\|\nu\|^\theta\left|\frac{1}{\mdel(\nu)}-\frac{1}{-\|\nu\|^2}\right| \le C,\qquad\text{for all $\nu\ne 0$ and all $\delta\in(0,1]$}.
\end{equation*}
\end{lemma}
\begin{proof}
We begin with the proof for $\delta=1$.  From Theorem~\ref{thm:asymptotic-near-zero} we have that, near $\nu=0$,
\begin{equation*}
\begin{split}
\|\nu\|^\theta\left(\frac{1}{\mdelone(\nu)}-\frac{1}{-\|\nu\|^2}\right)
&= \frac{\|\nu\|^\theta}{-\|\nu\|^2}\left(\frac{-\|\nu\|^2}{\mdelone(\nu)}-1\right)\\
&= \frac{\|\nu\|^\theta}{-\|\nu\|^2}\left(\frac{n+2-\beta}{4(n+2)(n+4-\beta)}\|\nu\|^2 + O(\|\nu\|^4)\right)
\end{split}
\end{equation*}
showing boundedness near $\nu=0$.  The fact that $\theta\le 2$ combined with the asymptotic formulas in Theorem~\ref{thm:asymptotics} shows that the quantity is also bounded for $\nu$ sufficiently far from $0$.  For intermediate $\nu$, the fact that neither $\|\nu\|$ nor $\mdelone(\nu)$ vanish completes the proof.

Now, suppose that $\delta\in(0,1)$.  From~\eqref{eq:multiplier-general}, we see that
\begin{equation*}
\mdel(\nu) = \frac{1}{\delta^2}\mdelone(\delta\nu).
\end{equation*}
Thus,
\begin{equation*}
\|\nu\|^\theta\left(\frac{1}{\mdel(\nu)}-\frac{1}{-\|\nu\|^2}\right)
= \frac{\delta^2}{\delta^\theta}\|\delta\nu\|^\theta\left(\frac{1}{\mdelone(\delta\nu)}-\frac{1}{-\|\delta\nu\|^2}\right).
\end{equation*}
The result follows from the $\delta=1$ case and the fact that $\theta\le 2$.
\end{proof}
\begin{theorem}\label{thm:convergence_solutions}
Let $n\ge 1$ and $\beta \le n+2$. Suppose $f\in H^s(\T^n)$ satisfies $\hat{f}_0= 0$.  Let $u\in H^{s+2}(\T^n)$  be the solution of  the Poisson equation $\Delta u =f$, with $\hat{u}_0=0$, and for any $\delta>0$, let $\udel$ be the solution of the nonlocal Poisson equation $\Ldel \udel=f$ defined in $\T^n$, with $\hatudel_0 = 0$. Then
\begin{equation}\label{eq:limit_delta_1}
\lim_{\delta\rightarrow 0^+}\udel=u, \;\;\mbox{in } \;\;  H^{s'}(\T^n),
\end{equation}
where $s'=s + \max\{0,\beta-n\}$.
\end{theorem}
\begin{proof}
Since $\udel$ solves the nonlocal Poisson equation, its Fourier coefficients satisfy
\begin{equation}
\label{udel_fourier}
\hatudel_k = \frac{1}{\mdel(\nu_k)} \hat{f}_k, \;\mbox{ for }\; k\in \mathbb{Z}^n\setminus \{0\},
\end{equation}
and, similarly, the Fourier coefficients of $u$ satisfy
\begin{equation}
\label{u_fourier}
\hat{u}_k = \frac{1}{-\|\nu_k\|^2} \hat{f}_k, \;\mbox{ for }\; k\in \mathbb{Z}^n\setminus \{0\},
\end{equation}
where $\hat{f}_k$ are the Fourier coefficients of $f$.  Moreover, $\hat{u}_0=\hatudel_0 = 0$.

From Theorem~\ref{thm:poisson}, $u,\udel\in H^{s'}(\T^n)$. Using \eqref{udel_fourier} and \eqref{u_fourier} we see that
\begin{equation*}
\begin{split}
\|\udel-u \|^2_{H^{s'}(\T^n)} &= \sum_{0\ne k\in\mathbb{Z}^n}(1+\|k\|^2)^{s'}\|\hatudel_k - \hat{u}_k\|^2 \\
&= \sum_{0\ne k\in\mathbb{Z}^n}(1+\|k\|^2)^{s'} \left| \frac{1}{\mdel(\nu_k)} - \frac{1}{-\|\nu_k\|^2} \right|^2 \|\hat{f}_k\|^2\\
&= \sum_{0\ne k\in\mathbb{Z}^n}(1+\|k\|^2)^{s'-s} \left| \frac{1}{\mdel(\nu_k)} - \frac{1}{-\|\nu_k\|^2} \right|^2 (1+\|k\|^2)^{s}\|\hat{f}_k\|^2
\end{split}
\end{equation*}
Let $\theta=s'-s=\max\{0,\beta-n\}$.  Since $f\in H^s(\mathbf{T}^n)$, we can pass the limit as $\delta\to 0$ inside the summation provided the terms
\begin{equation*}
(1+\|k\|^2)^{\theta} \left| \frac{1}{\mdel(\nu_k)} - \frac{1}{-\|\nu_k\|^2} \right|^2
= \left(\frac{1+\|k\|^2}{\|\nu_k\|^2}\right)^{\theta} \|\nu_k\|^{2\theta}\left| \frac{1}{\mdel(\nu_k)} - \frac{1}{-\|\nu_k\|^2} \right|^2
\end{equation*}
are bounded uniformly for $k\ne 0$ and $\delta\in(0,1]$.  Since $\|\nu_k\|$ is bounded uniformly away from $0$ for $k\in\mathbb{Z}^n\setminus\{0\}$ and $(1+\|k\|)/\|\nu_k\|$ is bounded as $\|k\|\to\infty$, Lemma~\ref{lem:reciprocals-bounded} allows us to pass the limit in $\delta$ inside the summation.  The pointwise convergence results of Corollary~\ref{cor:delta-beta} then complete the proof.
\end{proof}
\begin{lemma}
\label{lem:beta}
Let $\beta'<\beta<n+2$. Then, for all $\nu\ne 0$,
\[
\mdel<m^{\delta,\beta'}.
\]
\end{lemma}
\begin{proof}
Let $\nu\ne 0$. Then by a change of variables $w=\frac{1}{\delta} z$,  the integral form of the multipliers \eqref{eq:multiplier-cosine} becomes
\begin{eqnarray*}
	\mdel(\nu) &=& \cdel \delta^{n-\beta}\int_{B_1(0)}\frac{\cos(\nu\cdot \delta w)-1}{\|w\|^\beta}\;dw\\
	&=& \frac{2\Gamma\left(\frac{n}{2}+1\right)}
    {\pi^{n/2}\delta^{2}}\; (n+2-\beta)\int_{B_1(0)}\frac{\cos(\nu\cdot \delta w)-1}{\|w\|^\beta}\;dw,
\end{eqnarray*}
where we used \eqref{eq:cdel-explicit} in the last equation. Observing that $n+2-\beta<n+2-\beta'$ and $\|w\|^\beta<\|w\|^{\beta'}$, for $w\in B_1(0)$, it follows that 
\begin{eqnarray*}
	\mdel(\nu) &<& 
	\frac{2\Gamma\left(\frac{n}{2}+1\right)}
    {\pi^{n/2}\delta^{2}}\; (n+2-\beta')\int_{B_1(0)}\frac{\cos(\nu\cdot \delta w)-1}{\|w\|^{\beta'}}\;dw=m^{\delta,\beta'}(\nu).
\end{eqnarray*}

\end{proof}
\begin{theorem}\label{thm:convergence_solutions_beta}
Let $n\ge 1$ and $\delta>0$. Suppose $f\in H^s(\T^n)$ satisfies $\hat{f}_0= 0$.  Let $u\in H^{s+2}(\T^n)$  be the solution of  the Poisson equation $\Delta u =f$, with $\hat{u}_0=0$, and for any $\beta< n+2$, let $\udel$ be the solution of the nonlocal Poisson equation $\Ldel \udel=f$ defined in $\T^n$, with $\hatudel_0 = 0$. Then, for any $0<\epsilon<2$,
\begin{equation}\label{eq:limit_beta}
\lim_{\beta\to(n+2)^{-}}\udel=u, \;\;\mbox{in } \;\;  H^{s'}(\T^n),
\end{equation}
where $s'=s+2-\epsilon$.
\end{theorem}

\begin{proof}
For a given $0<\epsilon<2$, define $\beta'=n+2-\epsilon>n$.  For any $\beta\ge\beta'$, Theorem~\ref{thm:poisson} implies that $\udel\in H^{s+\beta-n}(\T^n)\subset H^{s'}(\T^n)$.  Moreover, $u\in H^{s+2}(\T^n) \subset H^{s'}(\T^n)$, so we can make sense of the limit statement in~\eqref{eq:limit_beta} for $\beta$ sufficiently close to $n+2$.

Proceeding as in the proof of Theorem~\ref{thm:convergence_solutions}, we have
\begin{equation}\label{eq:convergence-beta-key-est}
\|\udel-u \|^2_{H^{s'}(\T^n)} \le
\sum_{0\ne k\in\mathbb{Z}^n}(1+\|k\|^2)^{s'-s} \left| \frac{1}{\mdel(\nu_k)} - \frac{1}{-\|\nu_k\|^2} \right|^2 (1+\|k\|^2)^{s}\|\hat{f}_k\|^2.
\end{equation}
And, again, since $f\in H^s(\T^n)$, we can pass the limit in $\beta$ inside the sum provided the terms
\begin{equation}\label{eq:converge-beta-term}
(1+\|k\|^2)^{2-\epsilon}\left| \frac{1}{\mdel(\nu_k)} - \frac{1}{-\|\nu_k\|^2} \right|^2
\end{equation}
are uniformly bounded for $k\in\mathbb{Z}^n\setminus\{0\}$ and $\beta\in[\beta',n+2)$.  Applying Lemma~\ref{lem:beta} shows that
\begin{equation}\label{eq:converge-beta-term-est}
\begin{split}
\left| \frac{1}{\mdel(\nu_k)} - \frac{1}{-\|\nu_k\|^2} \right|^2
& \le
\frac{1}{\mdel(\nu_k)^2}
+ \frac{2}{|\mdel(\nu_k)|\|\nu_k\|^2}
+ \frac{1}{\|\nu_k\|^4} \\
&\le
\frac{1}{\mdelbp(\nu_k)^2}
+ \frac{2}{|\mdelbp(\nu_k)|\|\nu_k\|^2}
+ \frac{1}{\|\nu_k\|^4}.
\end{split}
\end{equation}
None of the denominators in the final sum vanish for $k\ne 0$.  Thus, Theorem~\ref{thm:asymptotics} shows that there exist constants $C_1,C_2>0$ so that
\begin{equation}\label{eq:converge-beta-m-bound}
|\mdelbp(\nu_k)| \ge C_1(1+\|k\|^2)^{\frac{2-\epsilon}{2}}\quad\text{and}\quad
\|\nu_k\|^2 \ge C_2(1+\|k\|^2)
\end{equation}
for all $k\in\mathbb{Z}^n\setminus\{0\}$.  Combining~\eqref{eq:converge-beta-term},~\eqref{eq:converge-beta-term-est} and~\eqref{eq:converge-beta-m-bound} shows that for all $k\ne 0$ and all $\beta\in[\beta',n+2)$,
\begin{equation*}
(1+\|k\|^2)^{2-\epsilon}\left| \frac{1}{\mdel(\nu_k)} - \frac{1}{-\|\nu_k\|^2} \right|^2
\le \frac{1}{C_1^2}
+ \frac{2}{C_1C_2(1+\|k\|^2)^{\epsilon/2}}
+ \frac{1}{C_2^2(1+\|k\|^2)^{\epsilon}},
\end{equation*}
establishing the uniform bound needed to pass the limit in $\beta$ through the sum in~\eqref{eq:convergence-beta-key-est}.
\end{proof}

\section{Conclusion}
This work provides a new representation for the Fourier multipliers of a nonlocal Laplace operator in terms of  a hypergeometric function. This  representation allows for a unified approach for the analysis  of the multipliers and, consequently, the nonlocal Laplacian and associated nonlocal equations. In particular, the new representation of the multipliers  is exploited  to  characterize their asymptotic behavior and show the convergence of the multipliers $\mdel$ to the multipliers of the Laplacian for two types of limits; $\delta\rightarrow 0$ and $\beta\rightarrow n+2$. In addition, we show the convergence of $\Ldel\rightarrow \Delta$ as $\beta\rightarrow n+2$ and identify the limiting behavior of $\Ldel$ when $\beta\rightarrow -\infty$. Moreover, the hypergeometric representation of the multipliers allows for extending the definition of the nonlocal Laplacian to a pseudo-differential operator when $\beta>n+2$.  Future work will further discuss this pseudo-differential operator and its physical interpretation.

Our new approach for the analysis of the Fourier multipliers of the nonlocal Laplacian is applied in the periodic setting to show a regularity result for the nonlocal Poisson equation in $\Rn$ and to show the convergence of its solution to the solution of the classical Poisson equation. Computational aspects of nonlocal equations based on the Fourier multipliers approach and generalizing the results of this work to the peridynamics theory of nonlocal mechanics will be addressed in future works.

\bibliographystyle{acm}
\bibliography{refs}
\end{document}